\documentclass[preprint,sort&compress,final,12pt]{elsarticle}
\usepackage{amsthm}
\usepackage{amsmath}
\usepackage{amssymb}
\usepackage{graphicx}
\usepackage{latexsym}
\usepackage{epstopdf}
\usepackage{natbib}
\usepackage{color}
\usepackage{hyperref}
\usepackage{mathrsfs}
\usepackage{verbatim}
\usepackage{geometry}

\newtheorem{theorem}{Theorem}[section]

\newtheorem{corollary}[theorem]{Corollary}

\newtheorem{definition}[theorem]{Definition}
\newtheorem{lemma}[theorem]{Lemma}
\newtheorem{proposition}[theorem]{Proposition}
\newtheorem{remark}[theorem]{Remark}
\numberwithin{equation}{section}

\def\dsum{\displaystyle\sum}

\def\Im{\mathrm{Im}}

\def\N{{\mathbb N}}
\def\D{{\mathbb D}}
\def\Z{{\mathbb Z}}

\def\R{{\mathbb R}}
\def\C{{\mathbb C}}

\journal{\empty}
\date{}

\begin{document}

\begin{frontmatter}

\title{The existence of invariant curves of a kind of almost periodic twist mappings}

\author[au1]{Yingdu Dong}

\address[au1]{Laboratory of Mathematics and Complex Systems (Ministry of Education), School of Mathematical Sciences, Beijing Normal University, Beijing 100875, People's Republic of China}

\author[au1]{Xiong Li\footnote{ Partially supported by the NSFC (11971059). Corresponding author.}}

\ead[au1]{xli@bnu.edu.cn}

\begin{abstract}
In this paper, we are concerned with the existence of invariant curves of the planar twist mappings where the perturbations are almost periodic. As an application, the existence of almost periodic solutions and the boundedness of all solutions for pendulum-type equations with an almost periodic external force are proved.
\end{abstract}


\end{frontmatter}

\section{Introduction}

In this paper we are concerned with the existence of invariant curves of the following planar almost periodic mapping
\begin{equation}
\mathfrak{M}:\left\{
\begin{array}{lc}
x_1 = x+y+f(x,y), &  \\
 & (x,y)\in\R\times[a,b],\\
y_1 = y+g(x,y), &
\end{array}
\right.\label{twist}
\end{equation}
where the perturbations $f(x,y)$ and $g(x,y)$ are almost periodic in $x$ with an infinite dimensional frequency $\omega=(\omega_1,...,\omega_{\nu},...)_{\nu \in \N^+}$ which satisfies
a kind of Diophantine condition, and the mapping $\mathfrak{M}$ is assumed to possess the intersection property.

In this paper we assume that the perturbations $f$, $g$ admit rapidly converging Fourier series expansions. To be more specific, we have to use the notations which are given later in the next section. For instance, assume that the perturbation $f$ takes the form
\begin{equation*}
f(x,y)= \sum_{l\in Z^{\N^+}} f_l(y) e^{i(\omega,l)x},
\end{equation*}
where the frequency $\omega$ satisfies a kind of Diophantine condition described in (\ref{Dio}) and the coefficient $f_l$ decays rapidly when the norm of $l$ goes to infinity. In fact we will use the weighted norm  $\|f\|_{r,s}$ given in Definition $\ref{func sp2}$ to describe the rate of decay. Meanwhile, the intersection property of $\mathfrak{M}$ means that the image of an almost periodic curve under the mapping $\mathfrak{M}$ will intersect with itself. Under these assumptions, if the perturbations are small enough in the sense of the $\|\cdot\|_{r,s}$ norm, the existence of invariant curves of the mapping will be proved. Based on this result, the existence of almost periodic solutions and the boundedness of all solutions for the pendulum-type equation (\ref{pendul}) with an almost periodic external force is also proved.

When the mapping in (\ref{twist}) is periodic with respect to the variable $x$, Moser had proved the existence of invariant closed curves under the smoothness assumption of $C^{333}$ in \cite{Moser}. And an analytic version theorem was presented in the book \cite{moser and siegel}. A version in class $C^5$ is presented in R\"{u}ssmann \cite{russmann} and the optimal version in class $C^p$, where $p>3$ is presented in Herman \cite{Herman1}, \cite{Herman2}. Moreover, Ortega established a variant of small twist theorem in \cite{ortega1} and also studied the existence of invariant curves of mappings with average small twist in \cite{ortega2}.

When the perturbations are quasi-periodic, there are also many results about the existence of invariant curves under the exact symplectic condition, see Zharnitsky \cite{Zhar}, or the reversible condition, see Liu \cite{liubin}. One can find a detailed summary in Huang \cite{huangpeng2}, \cite{huangpeng3}, \cite{huangpeng}.

Recently, Huang, Li and Liu had proved the existence of invariant curves when the perturbations are almost periodic and the mapping has intersection property in \cite{huangpeng2}. As an application they studied the existence of almost periodic solutions and the boundedness of all solutions for the superlinear Duffing's equation. The theorem proved by Huang is based on the discussions in \cite{spatial}, where P\"{o}schel had studied the existence of full dimension invariant tori for the infinite dimensional near-integrable Hamiltonian system. In that paper P\"{o}schel had developed the conception of ``spatial structures'', which is used to construct related weight functions to describe the decay rate of the coefficients of the perturbations. As described in \cite{spatial}, the perturbations are considered as composites of smaller pieces reflecting the underlying spatial structure rather than a single chunk. Huang used the conception of  ``spatial structures'' to describe the decay rate of the Fourier coefficients of the perturbations and gave the related nonresonance condition in his work. It is known that in KAM-type theorems, a good balance should be made between the nonresonance conditions for the small divisor and the decay rate of the coefficients of the perturbations. That is to say, if the related weight functions are too ``light'', there will be no frequency vectors satisfying the related nonresonance condition, while if the weights are too ``heavy'', the coefficients have to decay too fast thus there will be less perturbations to deal with. In addition to P\"{o}schel's work, Bourgain \cite{bour} had proved the existence of full dimension invariant tori for an infinite dimensional Hamiltonian system which is related to the one dimension Schr\"{o}dinger equation under an explicit decay rate and nonresonance condition, there is also a generalization in \cite{cong} by Cong, Liu, Shi and Yuan.

Inspired by the works of Huang and Bourgain, in this paper we will investigate the existence of invariant curves for almost periodic twist mappings when the decay rate condition and the nonresonance condition are more concrete and quantitative. As described in the next section, the weighted norm of the perturbations are simple and explicit. Roughly speaking, the coefficients of the perturbations should decay at least exponentially. For example, the perturbation may take the following form
\begin{equation*}
  \sum_{n=1}^{\infty} \frac{\epsilon}{2^n} \cos(\omega_n x),\end{equation*}
where $\omega_n$ is the component of the frequency vector and $\epsilon>0$ is small enough.

The detailed statement of the existence theorem will be put in the next section. The rest of the paper is organized as follows. In Section 2, firstly we list some basic properties of almost periodic functions which will be used later, secondly by the measure estimate the existence of the Diophantine frequency vector will be proved. Then we can give the definition of real analytic almost periodic functions and the related weighted norms. Later we introduce some properties of them which are followed by some details of the intersection property. At last the main invariant curve theorem (Theorem \ref{main}) for the almost periodic mapping $\mathfrak{M}$ given by (\ref{twist}) is stated. The proof of Theorem \ref{main} is given in Sections 3 and 4. In Section 5, as an application, we will prove the existence of almost periodic solutions and the boundedness of all solutions for the pendulum-type equation (\ref{pendul}) with an almost periodic external force.

\section{Notations and preliminaries}
In this section, some notations, definitions and properties which will be used in the proof of the main theorem are listed. The proofs of some well-known consequences are omitted.
\subsection{The module of almost periodic functions}
Firstly, the definition of almost periodic functions is stated as follows.
\begin{definition}
The function $f:\R \rightarrow \C$ is called almost periodic if for any sequence $\{t_n\} \subseteq \R$, there exists a subsequence $\{t_{n_k} \}$ such that $\{f(t+t_{n_k})\}$ converges uniformly in $\R$.
\end{definition}
There is also an equivalent definition.
\begin{definition}
$f:\R \rightarrow \C$ is almost periodic if and only if for all $\epsilon>0$, there exists a $l>0$ such that for all $x \in \R$, there exists a $ \tau \in (x,x+l)$ such that  $|f(t+\tau)-f(t)|<\epsilon$ is valid for all $t\in\R$.
\end{definition}
It is well-known that if $f$ is almost periodic then the limit
\begin{align*}
\lim_{T \rightarrow \infty} \frac{1}{T} \int_a ^{a+T} f(t)e^{-i\lambda t} dt
\end{align*}
exists for any $\lambda$ and $a$, in fact the limit converges uniformly in $a$. The set consisted of all $\lambda$ such that the limit above is not equal to zero is called the exponent set of $f$, which is countable. Thus the module of $f$ can be defined as the smallest additive group generated by all exponents of $f$.

Similarly the module $\mathcal{M}(f,g)$ is defined as the smallest additive group generated by the union of all exponents of $f$ and $g$, which can be represented as $\{\sum_{\nu} j_{\nu}\lambda_{\nu}:j_{\nu}\in\Z\}$, where $\lambda_{\nu}$ is an exponent of $f$ or $g$ and the sum is finite. Take a minimal generator of $\mathcal{M}(f,g)$ denoted by $\Lambda$ such that any element in $\mathcal{M}(f,g)$ is a finite linear combination of elements in $\Lambda$ with integral coefficients.

We remark that $\Lambda$ is countable, which can be deduced from that $\mathcal{M}(f,g)$ is countable. Meanwhile, $\Lambda$ may not be the maximal linear independent subset of $\mathcal{M}(f,g)$ with respect to integral coefficients, because any element in $\mathcal{M}(f,g)$ is a finite linear combination of the elements in this subset with rational coefficients, rather than integral coefficients.

\subsection{The hull of almost periodic functions}
For all real sequences $\{t_k\}$ such that $\{f(t+t_k)\}$ converges uniformly in $t$, the set of corresponding limit functions is called the hull of $f$. Equivalently, the hull ${\rm H}(f)$ is the closure of $\{f_t(\cdot)=f(\cdot+t):t\in\R\}$ in the sense of sup-norm. It is obvious that $f$ itself is in its hull, and the hull is a compact space due to the definition of almost periodic functions.

Denote the element in ${\rm H}(f)$ by $\xi$ and ${\rm H}(f)$ can be also regarded as the hull of any $\xi$ in ${\rm H}(f)$. For simplicity, denote $f$ by $\xi_0$ and denote $\xi(x+t)$ by $(\xi \cdot t)(x)$. It is well-known that there exists a compact Abelian group structure on ${\rm H}(f)$ with the identity $\xi_0$ which is generated by the flow  $\{\xi_0\cdot t\}$. More specifically, assume that
\begin{align*}
\xi_1=\lim_{n\rightarrow\infty} \ \xi_0 \cdot t_n,~ \xi_2=\lim_{n\rightarrow\infty} \xi_0 \cdot s_n
\end{align*}
are two elements of ${\rm H}(f)$, then the product is defined as
\begin{align*}
\xi_1 \cdot \xi_2=\lim_{n\rightarrow\infty} \ \xi \cdot (t_n+s_n).
\end{align*}
The limit above is well defined and the product is commutative. Moreover, define $\xi^{-1}$ by
\begin{align*}
\xi^{-1}=\lim_{n\rightarrow\infty}\ \xi \cdot (-t_n).
\end{align*}
As a consequence, the real axis $\R$ can be viewed as an embedded subgroup of ${\rm H}(f)$ if identify any real number $t$ with the element $\xi_0 \cdot t$.

\subsection{Module containment}
 In this subsection some well-known properties of almost periodic functions are given, and a slightly modified proposition about  module containment is proved.
 \begin{definition}
 	The $\epsilon-$translation set $T_{\epsilon}(f)$ of $f$ is defined as $\{\tau:|f(t+\tau)-f(t)|<\epsilon,$ for all $t \in \R\}$.
 \end{definition}

 The relation between $\epsilon-$translation set and the exponents is stated as in the following lemmas.
 \begin{lemma}
 	For any finite subset of the exponents of $f$, $\{\lambda_n:n=1,...,N\}$$\subseteq\mathcal{M}(f)$ and for all $ \epsilon>0$, there exists a $\delta > 0$, such that  $ T_{\delta}(f) \subseteq \{\tau:|\lambda_n\tau|<\epsilon \ \mod(2\pi)$, $n = 1,...,N\}$.\label{translation1}
\end{lemma}

\begin{lemma}
	For all $ \epsilon>0$, there exists a subset of the exponents of $f$, $\{\lambda_n:n=1,...,N\}\subseteq\mathcal{M}(f)$ and a positive $\delta$, such that $\{\tau:|\lambda_n\tau|<\delta \mod(2\pi),$ $n=1,...,N\}\subseteq T_{\epsilon}(f)$.\label{translation2}
\end{lemma}

\begin{theorem}[Kronecker theorem]
	For real numbers $\{\lambda_i\}$, $\{\theta_i\}$, where $i=1,...,n$, the two statements below are equivalent.
	
(i)	For all $\epsilon>0$, there exists $\tau \in \R$ such that
\begin{equation*}
|\lambda_i\tau-\theta_i|<\epsilon\ \mod(2\pi),\ \ i=1,...,n.
\end{equation*}

(ii) Assume that $\{k_i: i = 1,...,n\} \subseteq \Z$ satisfies $\dsum_{i=1}^{n} k_i\lambda_i = 0$, then \begin{equation*}
\sum_{i=1}^{n} k_i\theta_i = 0 \ \mod(2\pi).
\end{equation*}
\label{kronecker}
\end{theorem}

It is well-known that $\mathcal{M}(g)\subseteq\mathcal{M}(f)$ is equivalent to that for any sequence $\{t_n\} \subseteq \R$ such that $\{f(t+t_n)\}$ converges uniformly, $\{g(t+t_n)\}$ will also converge uniformly. And here a slightly modified statement is proved.
\begin{proposition}
	$\mathcal{M}(h)\subseteq\mathcal{M}(f,g)$ is equivalent to that assume $\{f(t+t_n)\}$ and $\{g(t+t_n)\}$ converge uniformly, then $\{h(t+t_n)\}$ converges uniformly.
\label{modified module containment}\end{proposition}
\begin{proof}
	See Appendix \ref{1} for the detailed proof.
\end{proof}
\begin{corollary}
	If $f$ and $g$ are two almost periodic functions, then $h(t)=f(t+g(t))$ is also almost periodic. Moreover, $\mathcal{M}(h)\subseteq \mathcal{M}(f,g)$.
\end{corollary}

\subsection{Existence of Diophantine frequencies}
The infinite dimensional frequency vector $\omega=(\omega_1,...,\omega_{i},...)$, where $|\omega_i|\leq 1$, $i \in \N^+$, is called Diophantine if  for any finitely supported non-zero integer vector $l = (l_1,...,l_{i},...)$,
\begin{equation}
 |(\omega,l)|= \Big\lvert \sum_{i}l_{i}\omega_{i} \Big\rvert > \frac{\gamma_0}{\prod_{i=1}^{\infty}(1+i^{1+\mu}\cdot|l_i|^{1+\mu})},
 \label{Dio}
\end{equation}
 where $\gamma_0$ and $\mu$ are two positive constants. Such frequency vector exists if $\gamma_0$ is small enough.
\begin{proposition}
	For any $\mu>0$, if $\gamma_0$ is small enough, then there exists $\omega$ satisfying (\ref{Dio}).
\end{proposition}
\begin{proof}
See Appendix \ref{2} for the detailed proof.
\end{proof}
For simplicity, denote the set of infinite dimensional integer  vectors  with finite support by $Z^{\N^+}$, and denote $Z^{\N^+}-\{0\}$ by $Z^{\N^+}_{\star}$.
\begin{proposition}
	For any interval $[a,b] \subseteq \R$, if $\gamma$ is small enough, then there exists $\alpha \in [a,b]$ such that for any $l \in Z^{\N^+}_{\star}$ and $n \in \Z$
	\begin{equation}
	\left\lvert (\omega,l)\frac{\alpha}{2\pi}-n \right\rvert> \gamma \prod_{i=1}^{\infty} \frac{1}{1+i^{2+2\mu}\cdot|l_i|^{2+2\mu}},\label{Dio2}
	\end{equation}
where $\omega$ and $\mu$ are as in (\ref{Dio}).
\end{proposition}
\begin{proof}
See Appendix \ref{3} for the detailed proof.
\end{proof}

In P\"{o}schel and Huang's papers, the authors consider $\mathcal{S}$ as a family of finite subsets $A$ in $\mathbb{Z}$ and the union of $\{A\}$ covers $\Z$. $\mathcal{M}$ reflects a spatial structure on $\mathbb{Z}$. In fact, for any $A \in \mathcal{S}$, $[A]$ gives a positive number which is called the weight of $A$. Also for an integer vector $k$ with finite support, $[[k]]$ is defined as $$[[k]] = \min\limits_{\mbox{supp}\, k \subseteq A \in \mathcal{S}}[A].$$ It can be proved that there exist an approximation function $\Delta_0$ and a probability measure $\mu$ on the parameter space $\mathbb{R}^{\mathbb{Z}}$ with support at any prescribed point such that the measure of the set of $\omega$ satisfying the following inequalities
\begin{equation*}
\begin{array}{ll}
| k,\omega | \geq \frac{c}{\Delta_0([[k]])\Delta_0(|k|)},\ \ \ \ c>0, ~\mbox{for all}~ k\neq 0 \in \mathbb{Z}_{\mathcal{S}}^{{\mathbb{Z}}},
\end{array}
\end{equation*}
is positive for a suitably small $c$, where $|k|=\sum \limits_{\lambda \in \mathbb{Z}} |k_{\lambda}|.$ The nonresonance condition for $\alpha$ is similarly defined.

The approximation function $\Delta_0$ here is somewhat abstract and in fact its form may be somewhat complex, for details see P\"{o}schel \cite{spatial} or Huang \cite{huangpeng2}, thus it may be somewhat difficult to check whether the nonresonance condition is satisfied. And the measure $\mu$ is a kind of product Gaussian measure, while the measure estimate in this paper is based on the usual product Lebesgue measure.

\subsection{The space of analytic almost periodic functions}

In this subsection, the space of the real analytic almost periodic functions which admit rapidly converging Fourier series expansions is defined.
\begin{definition}\label{func sp1}
	Assume $\omega$ is a fixed frequency vector satisfying (\ref{Dio}), $AP_r(\omega)$ is defined as the set consisted of all functions $$f:\{z:|\Im~z|<r\} \rightarrow \C,$$ which admit the form $f(z)=\sum_{l\in Z^{\N^+}} f_l e^{i(\omega,l)z}$ and satisfies $$\|f\|_r := \sum_{l\in Z^{\N^+}} |f_l| e^{r\|l\|} < \infty,$$ where $f_l\in \C$ and $\|l\|$ is defined as $\sum_{i=1}^{\infty} |l_i|\cdot i$.
\end{definition}
\begin{remark}
	It is clear that if the sum above is finite, then there are at most countable $l\in \Z^{\N^+}$ such that $f_l \neq 0$. Besides, the space $AP_r(\omega)$ is a complex Banach space equipped with the norm $\|f\|_r$.  Moreover, the function $f$ in $AP_r(\omega)$ is analytic in $z$ and when $z$ takes real values $f$ is almost periodic in $z$. The exponents of $f$ are $\{(\omega,l)=\sum_{i}l_i\omega_i: |\mbox{supp} ~ l | < \infty ,f_l \neq 0\}$.
\end{remark}

Assume $x$ and $y$ are complex variables, then the function space $AP_{r,s}(\omega)$ is defined similarly.
\begin{definition}\label{func sp2}
	$AP_{r,s}(\omega)$ is defined as the set of all functions $$f:\{x:|\Im~x|<r\}\times B_s(y_0) \rightarrow \C,$$ which admit the form $f(x,y)=\sum_{l\in Z^{\N^+}} f_l(y) e^{i(\omega,l)x}$ and satisfies  $$\|f\|_{r,s} = \sum_{l\in Z^{\N^+}} |f_l(y)|_s e^{r\|l\|} < \infty,$$ where $f_l$ are analytic functions in $y$, $|f_l|_s$ denotes the sup-norm
	  and $\|l\|$, $\omega$ are the same as in Definition \ref{func sp1}.
\end{definition}

It is obvious that $f(x,y)$ is analytic in $x,y$ and for any fixed $y$, $f(x,y)$ is almost period in real $x$, and $\mathcal{M}(f) = \mbox{span}\{\omega_i:i=1,...,k,... \}$.  Also there is a useful estimate for the product of two functions in $AP_{r,s}(\omega)$ as follows.
\begin{lemma}\label{homogeneous}
Assume $f,g \in AP_{r,s}(\omega)$, then $\|fg\|_{r,s} \leq \|f\|_{r,s}\|g\|_{r,s}$.
\end{lemma}
The proof is simple thus we omit it.
And some important properties of real analytic almost periodic functions are given as follows and the detailed proofs can be found in Appendix \ref{appen}. For simplicity, denote $D(r,s)$ the region $\{(x,y)\in \C^2: |\Im~x|<r, |y-\alpha|<s\}$, where $\alpha$ is as in (\ref{Dio2}) and denote $\bigcup_{r,s}AP_{r,s}(\omega)$ by $AP(\omega)$.

\begin{lemma}\label{lemma of function space}
	The following statements are true:\\
	(i) Assume $f(x,y)$, $g(x,y)$ $\in AP(\omega)$, then $f\pm g \in AP(\omega)$.\\
	(ii) Assume $f(x,y)$, $u(x,y)$ and $v(x,y)$ $\in AP(\omega)$, then $f(x+u,y+v) \in AP(\omega)$, in fact the norm of $f(x+u,y+v)$ can be estimated in a smaller domain. \\
	(iii) Assume $u(x,y)$ and $v(x,y)$ $\in AP_{r,s}(\omega)$, which are small enough in the sense of the $\|\cdot\|_{r,s}$ norm, then the mapping
\begin{equation*}
\mathfrak{U}:\left\{
\begin{array}{lc}
x = \xi + u(\xi,\eta), &  \\
y = \eta + v(\xi,\eta), &
\end{array}
\right.
\end{equation*}
is revertible in a smaller domain and the inverse mapping takes the form
\begin{equation*}
\mathfrak{U}^{-1}:\left\{
\begin{array}{lc}
\xi = x + u'(x,y), &  \\
\eta = y + v'(x,y),&
\end{array}
\right. \label{inverse mapping}
\end{equation*}
where $u'$, $v'$ belong to $AP_{r',s'}(\omega)$ for $r'<r$ and $s'<s$.
\end{lemma}
\begin{proof}
	See Appendix \ref{4} for the detailed proof.
\end{proof}
\begin{remark}
	The estimates in (ii) is useful in this paper because the weighted norm is sensible under coordinate transformations, not the same as the sup-norm.
\end{remark}

\subsection{The intersection property}
\begin{definition}
	Let $\mathfrak{M}$ be the mapping given by (\ref{twist}). $\mathfrak{M}$ is said to have the intersection property if
	$$
	\mathfrak{M}(\Gamma) \cap \Gamma \neq \emptyset$$
	for any real analytic almost periodic curve $\Gamma : y = \phi(x)$, where $\phi$ belongs to $AP(\omega)$.
\end{definition}

It had been proved in Huang \cite{huangpeng2} that assume the mapping is exact symplectic, then the intersection property is valid, see also Ortega \cite{exact sym}. Roughly speaking, consider the domain $\mathbb{B}$ bounded by four curves: $\{x=t\}$, $\{x=T\}$, $\{y=r\}$ and $\Gamma$, and the domain  $\mathbb{B}_1$ bounded by four curves:  $\{x=t\}$, $\{x=T\}$, $\{y=r\}$ and $\mathfrak{M}(\Gamma)$, where $\{y=r\}$ is below $\Gamma$ and $\mathfrak{M}(\Gamma)$ when the angel variable $x \in [t,T]$. Under the assumption of exact symplectic condition, there exists a function $H(x,y) \in AP(\omega)$ such that $dH = y_1dx_1 - ydx$. Assume $f$, $g$ are small enough and $\phi$ is close enough to a constant, then from Lemma \ref{lemma of function space}, $x=x_1+w(x_1)$, where $w \in AP(\omega)$. By Green formula the difference of the area of the two domains will be
 \begin{equation*}
 \begin{aligned}
 \Delta(T)-\Delta(t)= &H(T+w(T),\phi(T+w(T))+g(T+w(T),\phi)) \\ &-H(t+w(t),\phi(t+w(t))+g(t+w(t),\phi)) \\ &+ \int_{t+w(t)}^{T+w(T)} \phi(x)dx \\
 &-\int_{t}^{T} \phi(x) dx .
 \end{aligned}
\end{equation*}
Thus $\Delta$ is almost periodic. Then from the property of almost periodic functions there exist $T$ and $t$ $\in \R$ such that the difference of the area is 0 and the intersection point exists.

\subsection{The main theorem}

\begin{theorem}
	Assume that the almost periodic mapping $\mathfrak{M}$ given in (\ref{twist}) has the intersection property, and $f$, $g \in AP_{r_0,s_0}(\omega)$ with $\omega$ satisfying the nonresonance condition (\ref{Dio}) and $\alpha$ satisfying (\ref{Dio2}),	
	then for any $\epsilon>0$, there exists  $\delta = \delta(\epsilon, r_0, s_0, \gamma)>0$ such that if
$$	\|f\|_{r_0,s_0} + \|g\|_{r_0,s_0} < \delta, $$
	then there is an invariant curve $\Gamma$:
	$$
	x = \xi + u(\xi), \ \
	y = v(\xi),
	$$
	where $u$, $v \in AP_{r_0/2}(\omega)$ satisfying $\|u\|_{r_0/2} + \|v-\alpha\|_{r_0/2} < \epsilon$.
	Moreover, the restriction of $\mathfrak{M}$ onto $\Gamma$ is
	~$\mathfrak{M}|_{\Gamma} : \xi_1 =\xi + \alpha$. \label{main}
\end{theorem}

In many applications, the small twist mapping
\begin{equation*}
\mathfrak{M}_\delta:\left\{
\begin{array}{lc}
x_1=x+\delta y+ f(x,y,\delta), &\\
& (x,y) \in \R\times[a,b]\\
y_1=y+ g(x,y,\delta), &
\end{array}\right.
\end{equation*}
are met, where $0<\delta<1$,  $f(x,y,\delta)$, $g(x,y,\delta)$ $\in AP_{r,s}(\omega)$ for each $\delta>0$. Moreover, $$\|f(\cdot,\cdot,\delta)\|_{r,s} +\|g(\cdot,\cdot,\delta)\|_{r,s} \leq \delta \cdot \epsilon.$$

\begin{corollary}\label{small twist}
Choose the fixed $\gamma$, $\mu>0$ and $\epsilon$ small enough. Also choose $\alpha \in [a+\gamma,b-\gamma]$ satisfying $$\left|(l,\omega)\frac{\delta\alpha}{2\pi}-n\right|>\prod_{i=1}^{\infty}\frac{\gamma}{1+|l_i|^{2+2\mu}\cdot i^{2+2\mu}}$$ for any $n \in \Z$, $l \in Z^{\N^+}_{\star}$. Furthermore, assume the mapping $\mathfrak{M}_{\delta} $ satisfies the intersection property for any $\delta>0$.

Then for any fixed $\delta$, there exists an invariant curve under the mapping denoted by $$\Gamma_{\delta}: x=\xi+\phi(\xi,\delta), y=\psi(\xi,\delta),$$ where $\phi(\cdot,\delta)$, $\psi(\cdot,\delta) \in AP(\omega)$. Furthermore, the restriction of $\mathfrak{M}_{\delta}$ onto $\Gamma_{\delta}$ takes the form $\xi_1=\xi+\delta\alpha$.
\end{corollary}
In applications, usually $\delta$ tends to zero.

\section{The KAM step}
The proof of the main theorem is based on the KAM approach, it suffices to find a sequence of coordinate transformations which make the mapping more and more close to
\begin{equation*}
\left\{
\begin{array}{lc}
 x_1 = x + y,&\\[0.2cm]
 y_1 =y.&
\end{array}
\right.
\end{equation*}
\subsection{Construction of change of variables}
The form of transformation $\mathfrak{U}$ is assumed as
\begin{equation}
\mathfrak{U}:\left\{
\begin{array}{lc}
x = \xi + u(\xi,\eta),&\\[0.2cm]
y =\eta + v(\xi,\eta),&
\end{array}
\right.\label{transformation}
\end{equation}
where $u$ and $v$ $\in AP(\omega)$. Under this transformation, the original mapping $\mathfrak{M}$ is changed into
\begin{equation}
\mathfrak{U}^{-1}\mathfrak{M}\mathfrak{U} : \left\{
\begin{array}{lc}
\xi_1 = \xi + \eta + f_+(\xi,\eta),&\\[0.2cm]
\eta_1 =\eta + g_+(\xi,\eta),&
\end{array}
\right.\label{new eq}
\end{equation}
where $f_+$ and $g_+$ $\in AP(\omega)$ and are defined in a smaller domain $D(r_+,s_+)$, moreover, $\|f_+\|_{r^+,s^+}+\|g_+\|_{r^+,s^+} \ll \|f\|_{r,s}+\|g\|_{r,s}$. For simplicity, denote
$$\epsilon=\|f\|_{r,s}+\|g\|_{r,s},$$ where
$D(r,s)$ is defined in section 2 and $\delta := r-r^+$, $\rho : = s-s^+$.
From (\ref{transformation}) and (\ref{new eq}), we obtain
\begin{equation*}
\left\{
\begin{aligned}
&f_+(\xi,\eta) = f(\xi+u,\eta+v)+v(\xi,\eta)+u(\xi,\eta)-u(\xi+\eta+f_+,\eta+g_+),\\[0.2cm]
&g_+(\xi,\eta) = g(\xi+u,\eta+v)+v(\xi,\eta)-v(\xi+\eta+f_+,\eta+g_+).
\end{aligned}
\right.
\end{equation*}

To solve $u$ and $v$, as in the periodic case, we solve the linear equations
\begin{equation*}
\left\{
\begin{aligned}
&u(\xi+\alpha,\eta)-u(\xi,\eta)=v(\xi,\eta)+f(\xi,\eta),\\[0.2cm]
&v(\xi+\alpha,\eta)-v(\xi,\eta)=g(\xi,\eta).
\end{aligned}
\right.\label{linear eq}
\end{equation*}
The difference equations will introduce the small divisor, and it can be solved only if the mean value of $g$ over the first variable vanishes. For this reason, we consider the modified homological equations
\begin{equation}
\left\{
\begin{aligned}
&u(\xi+\alpha,\eta)-u(\xi,\eta)=v(\xi,\eta)+f(\xi,\eta),\\[0.2cm]
&v(\xi+\alpha,\eta)-v(\xi,\eta)=g(\xi,\eta)-[g](\eta),
\end{aligned}
\right.\label{linear eq2}
\end{equation}
where $[\ ]$ denotes the mean value of the function over the first variable.

\subsection{Estimates of u, v}

In order to solve $u$, $v$ from the equation (\ref{linear eq2}), it suffices to consider the following difference equation firstly
\begin{equation}
s(x+\alpha)-s(x)=h(x), \label{h eq}
\end{equation}
where $h \in AP_r(\omega)$.
\begin{lemma}\label{lemma hol}
	Suppose $h \in AP_r(\omega)$, where $\omega$ satisfies (\ref{Dio}) and $\alpha$ satisfies (\ref{Dio2}), then for any $r'$ satisfying $0<r'<r$, the difference equation (\ref{h eq}) has the unique solution $s\in AP_{r'}(\omega)$ with 0 mean value if and only if $h$ takes 0 mean value. In this case, there is the norm estimate
	\begin{equation*}
	\|s\|_{r'} \leq \gamma^{-1}\|h\|_re^{\delta^{-2}}.
	\end{equation*}
\end{lemma}
\begin{proof}
	From $h \in AP_r(\omega)$ and the 0 mean value condition, $h$ takes the form
	$$h(x)=\sum_{l\in Z^{\N^+}_{\star}} h_l e ^{i(\omega,l)x}.$$
	Assume
	$$s(x)=\sum_{l\in Z^{\N^+}_{\star}} s_l e ^{i(\omega,l)x},$$	
and by straight calculation, $s_l$ takes the form
$$s_l=\frac{h_l}{e^{i(\omega,l)\alpha-1}}.$$
Note that the mean value $h_0 = \lim_{T \rightarrow \infty} \int_{0}^{T} h(t) dt$ must be 0 to make sure the existence of $s$. From (\ref{Dio2}) and assume $s$ takes 0 mean value it follows that
\begin{equation*}
	\|s\|_{r'} \leq \gamma^{-1} \cdot \sum_{l\in Z^{\N^+}_{\star}} \bigg(\prod_{i=1}^{\infty} \Big(1+|l_i|^{2+\mu}\cdot i^{2+2\mu}\Big)\bigg)\cdot e^{-(r-r')\|l\|}\cdot |h_l| \cdot e^{r\|l\|}.
\end{equation*}
Further,
\begin{equation*}
\begin{aligned}
\|s\|_{r'} \leq& \gamma^{-1} \cdot \|h\|_r\ \cdot \sup_{l\in Z^{\N^+}_{\star}}\left\{\prod_{i=1}^{\infty}\Big(1+|l_i|^{2+\mu}\cdot i^{2+2\mu}\Big)\cdot e^{-\delta\|l\|}\right\}\\
\leq& \ \gamma^{-1} \cdot \|h\|_r \cdot \sup_{l\in Z^{\N^+}_{\star}}\left\{\prod_{i=1}^{\infty} \max\Big\{2|l_i|^{2+\mu}\cdot i^{2+2\mu} \cdot e^{-\delta |l_i|i},1\Big\} \right\}\\
\leq& \ \gamma^{-1} \cdot \|h\|_r \cdot \sup_{l\in Z^{\N^+}_{\star}}\Bigg\{\prod_{i=1}^{\big[\frac{1}{\delta^{1.5}}\big]+1}\max\left\{2|l_i|^{2+\mu}\cdot i^{2+2\mu} \cdot e^{-\delta |l_i|i},1\right\}  \\
&\times \prod_{i=\big[\frac{1}{\delta^{1.5}}\big]+1}^{\infty}\max\Bigg\{\frac{2e^{\frac{-1}{\sqrt{\delta}}}}{\delta^{1.5\cdot (2+2\mu)}},1\Bigg\}\Bigg\}\\
\leq& \ \gamma^{-1} \cdot \|h\|_r \cdot \sup_{l\in Z_{\N^+}^{\star}}\Bigg\{\prod_{i=1}^{\big[\frac{1}{\delta^{1.5}}\big]+1}\max\Big\{2|l_i|^{2+\mu}\cdot i^{2+2\mu} \cdot e^{-\delta l_ii},1\Big\}\Bigg\}\\
\leq& \ \gamma^{-1}\|h\|_re^{\delta^{-2}},
\end{aligned}
\end{equation*}
where $\delta=r-r'$ is small enough.
Then the conclusion is proved.
\end{proof}

Now the equation (\ref{linear eq2}) can be solved in a smaller domain from Lemma \ref{lemma hol}. Denote the solution to the second equation which has 0 mean value by $\widetilde{v}$. And in the first equation, to make the mean value over the first variable on the right side vanish, the mean value of $v$ must satisfy that
\begin{equation*}
[v](\eta)=-[f](\eta).
\end{equation*}
Then as a consequence, we have
\begin{equation*}
\|[v]\|_{r,s} \leq \|f\|_{r,s}.
\end{equation*}
Thus there is a solution $v(\xi,\eta) = \widetilde{v}-[f]$ to the second equation, and the first equation can be also solved. Specifically, denote $p(\xi,\eta) = v(\xi,\eta) + f(\xi,\eta)$ defined in $D(r-\delta/16,s)$, which satisfies
$$\|p\|_{r-\delta/16,s} \leq e^{\frac{c_1}{\delta^2}} \cdot \epsilon,$$ here $c_1$ depends only on $\mu$. Thus there is a solution $u(\xi,\eta)$ defined in $D(r-\delta/8,s)$ satisfying
\begin{equation*}
 \|u(\xi,\eta)\|_{r-\delta/8,s}+ \|v(\xi,\eta)\|_{r-\delta/8,s} \leq e^{\frac{c_2}{\delta^2}} \cdot \epsilon,
 \end{equation*}
 where $c_2$ depends only on $c_1$ and $\mu$.

For simplicity, denote $D^{i} = D(r-i\cdot \delta/4,s- i\cdot \rho/4)$. And by  Cauchy's estimate \ref{Cauchy}, if assuming that $\delta$ is small enough then the derivatives of $u$, $v$ in the domain $D^{(1)}$ satisfy that
\begin{equation}
\left \{
\begin{array}{lc}
\|u_{\xi}\|_{r-\delta/4,s-\rho/4}+\|v_{\xi}\|_{r-\delta/4,s-\rho/4} \leq e^{\frac{c_3}{\delta^2}}\cdot \epsilon,&\\
\\
\|u_{\eta}\|_{r-\delta/4,s-\rho/4}+\|v_{\eta}\|_{r-\delta/4,s-\rho/4} \leq \frac{1}{\rho} \cdot e^{\frac{c_3}{\delta^2}}\cdot \epsilon,&
\end{array}
\right.\label{condition 1}
\end{equation}
where $c_3>c_2$. Thus if \begin{equation*}
e^{\frac{c_3}{\delta^2}}\cdot \epsilon < \frac{\max\{\delta,\rho\}}{4},
\end{equation*}
then the mapping $\mathfrak{U}$ maps $D(r^+,s^+)$ into $D^{(3)}$. And if $\epsilon < \frac{\rho}{4}$ and $\epsilon + s - \frac{3\rho}{4} < \frac{\delta}{4}$, then the mapping $\mathfrak{M}$ maps $D^{(3)}$ to $D^{(2)}$. Moreover by the property (iii) of Lemma \ref{lemma of function space}, $\mathfrak{U} ^{-1}$ is well defined in $D^{(2)}$ and maps $D^{(2)}$ to $D^{(1)}$ when
\begin{equation}
\frac{1}{\rho}\cdot e^{\frac{c_3}{\delta^2}}\epsilon \cdot \exp\left(2e^{\frac{c_3}{\delta^2}}\epsilon+\big(s-\rho/2 + 2e^{\frac{c_3}{\delta^2}}\epsilon\big)\cdot (s-\rho/4)^{-1} \right)<1/2. \label{condition3}
\end{equation}

Then $\mathfrak{U}^{-1}\mathfrak{M}\mathfrak{U}$ is well defined in $D(r^+,s+)$, and in this case there is the estimate
 \begin{equation*}
 \|u'\|_{r-\delta/2,s-\rho/2}+\|v'\|_{r-\delta/2,s-\rho/2} \leq 2e^{\frac{c_3}{\delta^2}}\cdot \epsilon.
 \end{equation*}
And there is a rough estimate about $\|f_+\|_{r^+,s^+}+\|g_+\|_{r^+,s^+}$,
if denote
\begin{equation*}
2\epsilon \cdot \exp\left(e^{\frac{c_3}{\delta^2}}\epsilon+\big(s^+ + e^{\frac{c_3}{\delta^2}}\epsilon\big)/s\right) + 2e^{\frac{c_3}{\delta^2}}\cdot \epsilon\\
\end{equation*}
by $\Delta$, that is,
\begin{equation}
	\|f_+\|_{r^+,s^+}+\|g_+\|_{r^+,s^+}  \leq \Delta+ 2e^{\frac{c_3}{\delta^2}}\epsilon\cdot \exp\left(\Delta+ \frac{2s^+ +\Delta}{s-\rho/2}\right).
	\label{rough estimate}
\end{equation}
For simplicity, denote the right side of (\ref{rough estimate}) by $\widetilde{\Delta}$.
In the next subsection, a more sharp estimate of the new perturbations will be given.

\subsection{Estimates on the new perturbations}

Because $u$ and $v$ are solutions to the linear equation (\ref{linear eq2}), the right side of the equation
\begin{equation*}
\left\{
\begin{aligned}
&f_+(\xi,\eta) = f(\xi+u,\eta+v)+v(\xi,\eta)+u(\xi,\eta)-u(\xi+\eta+f_+,\eta+g_+),\\[0.2cm]
&g_+(\xi,\eta) = g(\xi+u,\eta+v)+v(\xi,\eta)-v(\xi+\eta+f_+,\eta+g_+),
\end{aligned}
\right.
\end{equation*}
can be transformed to
\begin{equation*}
\left\{
\begin{aligned}
f_+(\xi,\eta) =& (f(\xi+u,\eta+v)-f(\xi,\eta))\\[0.2cm]
&-(u(\xi+\eta+f_+,\eta+g_+)-u(\xi+\alpha,\eta)),\\
\\
g_+(\xi,\eta) =& (g(\xi+u,\eta+v)-g(\xi,\eta)+[g](\eta))\\[0.2cm]
&-(v(\xi+\eta+f_+,\eta+g_+)-v(\xi+\alpha,\eta)).
\end{aligned}
\right.
\end{equation*}
Moreover, by the mean value theorem, it is transformed to
\begin{equation*}
\left\{
\begin{aligned}
(1+u_{\xi})f_+ =& (f(\xi+u,\eta+v)-f(\xi,\eta))\\[0.2cm]
&-u_{\xi}\cdot(\eta-\alpha)-u_{\eta}\cdot g_+,\\
\\
(1+v_{\eta})g_+ =& (g(\xi+u,\eta+v)-g(\xi,\eta)+[g](\eta))\\[0.2cm]
&-v_{\xi}\cdot(\eta+f_+-\alpha).
\end{aligned}
\right.
\end{equation*}

In the following we will approximate the troublesome mean value $[g]$ by
\begin{equation*}
h(\eta)=g(\alpha)+[g]_{\eta}(\alpha)(\eta-\alpha).
\end{equation*}

By Cauchy's estimate and the property (ii) of Lemma \ref{lemma of function space}, one has the following estimates
\begin{equation*}
\begin{aligned}
\|f(\xi+u,\eta+v)-f(\xi,\eta)\|_{r^+,s^+} \leq &  \max\left\{\frac{1}{\rho},\frac{1}{\delta}\right\} \cdot e^{\frac{c_3}{\delta^2}}\cdot \epsilon^2\\
 &\times \exp \left(e^{\frac{c_3}{\delta^2}}\epsilon+\big(s^+ + e^{\frac{c_3}{\delta^2}\epsilon}\big)/s\right),
\end{aligned}
\end{equation*}
\begin{equation*}
\begin{aligned}
\|g(\xi+u,\eta+v)-g(\xi,\eta)\|_{r^+,s^+} \leq&   \max \left \{\frac{1}{\rho},\frac{1}{\delta}\right\}  \cdot e^{\frac{c_3}{\delta^2}}\cdot \epsilon^2 \\
& \times \exp \left(e^{\frac{c_3}{\delta^2}}\epsilon+\big(s^+ + e^{\frac{c_3}{\delta^2}\epsilon}\big)/s\right).\\
\end{aligned}
\end{equation*}

By the rough estimate on new perturbations (\ref{rough estimate}), we have
\begin{equation*}
\|u_{\xi}\cdot(\eta-\alpha)\|_{r^+,s^+} \leq
\exp\left(s^+ + \widetilde{\Delta}+ \frac{2s^+ + \widetilde{\Delta}}{s-\rho/4}\right)  \cdot e^{\frac{c_3}{\delta^2}}\epsilon  \cdot s^+ ,
\end{equation*}
\begin{equation*}
\|u_{\eta}\cdot g_+\|_{r^+,s^+} \leq
\exp\left(s^+ + \widetilde{\Delta}+ \frac{2s^+ + \widetilde{\Delta}}{s-\rho/4}\right) \cdot \frac{1}{\rho} \cdot e^{\frac{c_3}{\delta^2}}\epsilon  \cdot \widetilde{\Delta} ,
\end{equation*}
\begin{equation*}
\|v_{\xi}\cdot(\eta+f_+ -\alpha)\|_{r^+,s^+} \leq
\exp \left(s^+ + \widetilde{\Delta}+ \frac{2s^+ + \widetilde{\Delta}}{s-\rho/4}\right) \cdot e^{\frac{c_3}{\delta^2}}\epsilon  \cdot \left(s^+ + \widetilde{\Delta}\right) .
\end{equation*}
Moreover, we have
\begin{equation*}
\|[g]-h\|_{r^+,s^+} \leq \left(\frac{s^+}{s-s^+}\right)^2 \cdot \epsilon.
\end{equation*}
Finally we get
\begin{equation*}
\begin{aligned}
\|f_+\|_{r^+,s^+}+\|g_+ -h\|_{r^+,s^+} \leq& \exp\left(s^+ + \widetilde{\Delta}+ \frac{2s^+ + \widetilde{\Delta}}{s-\rho/4}\right)
 \cdot e^{\frac{c_3}{\delta^2}}\epsilon  \cdot \left(2s^+ + 4\widetilde{\Delta}/\rho\right) \\
 &+ \max\left \{\frac{1}{\rho},\frac{1}{\delta}\right\}  \cdot e^{\frac{c_3}{\delta^2}} \cdot \epsilon^2
  \cdot \exp \left(e^{\frac{c_3}{\delta^2}}\epsilon+\big(s^+ + e^{\frac{c_3}{\delta^2}}\epsilon\big)/s\right)\\
  &+\left(\frac{s^+}{s-s^+}\right)^2 \cdot \epsilon.
\end{aligned}
\end{equation*}

Assume the following conditions
\begin{equation}
\left \{
\begin{aligned}
&(a) \ s^+ < \frac{s}{3},\\
&(b)\ \rho < \delta,\\
&(c)\ e^{\frac{c_3}{\delta^2}}\epsilon < s \ll 1 ,\\
&(d)\ \epsilon + s^+ + \frac{\rho}{4} < \frac{\delta}{4},
\label{series condition}
\end{aligned}\right.
\end{equation}
are satisfied,  immediately we get
\begin{equation*}
\Delta \leq C e^{\frac{c_3}{\delta^2}} \cdot \epsilon,
\end{equation*}
\begin{equation*}
\widetilde{\Delta} \leq C' e^{\frac{c_3}{\delta^2}} \cdot \epsilon.
\end{equation*}
where $C$ and $C'$ are two positive constants.

Moreover
\begin{equation*}
\|f_+\|_{r^+,s^+}+\|g_+ -h\|_{r^+,s^+} \leq Q := c_4\left\{e^{\frac{2c_3}{\delta^2}}\cdot\left( \frac{\epsilon^2}{s}+s\epsilon\right)+\left(\frac{s^+}{s}\right)^2\epsilon \right\} .
\end{equation*}

The preliminary estimate for $h$ is insufficient. However, by the intersection property of $\mathfrak{M}$, the mapping $\mathfrak{M}^+:=\mathfrak{U}^{-1}\mathfrak{M}\mathfrak{U}$ also has the intersection property. Thus each curve $\eta=\eta_0$, where $\eta_0$ is constant, intersects with its image under $\mathfrak{M}^+$, thus there exists a point $(\xi_0,\eta_0) \in D(r,s)$ such that $g_+(\xi_0,\eta_0)=0$ for any real $\eta_0$ which satisfies $|\eta_0-\alpha| \leq s^+$. Then we have $\|h(\eta)\|_{r^+,s^+} \leq Q$. Specially, set $\eta=\alpha$, we get $|[g](\alpha)| \leq Q$, and set $\eta=\alpha+s^+$, we get \begin{equation*}\left|[g](\alpha)+[g]_{\alpha}(\alpha)s^+\right| \leq Q.\end{equation*}
Thus $\|h(\eta)\|_{r^+,s^+} = \sup_{|\eta-\alpha|<s^+} |h(\eta)| \leq 3Q$.

Finally we have
\begin{equation*}
\|f_+\|_{r^+,s^+}+\|g_+ \\|_{r^+,s^+} \leq 4Q \leq c_6\left\{e^{\frac{c_5}{\delta^2}}\cdot\left( \frac{\epsilon^2}{s}+s\epsilon\right)+\left(\frac{s^+}{s}\right)^2\epsilon\right\}.
\end{equation*}

\subsection{The iteration lemma}

To summarize the above discussions, the iteration lemma is given as follows.
\begin{lemma}
	Consider the mapping
	\begin{equation*}
	\mathfrak{M}:\left\{
	\begin{array}{lc}
	x_1 = x+y+f(x,y), &  \\
	& (x,y)\in\R\times[a,b],\\
	y_1 = y+g(x,y), &
	\end{array}
	\right.
	\end{equation*}
where $f$ and $g$ belong to $AP_{r,s}(\omega)$. Let
$$\epsilon = \|f\|_{r,s}+\|g\|_{r,s},$$
and $r^+<r<1$, $s^+<s<1$, $\delta= r-r^+$, $\rho = s-s^+$.
Assume $\epsilon$, $\delta$, $\rho$ satisfy (\ref{series condition}), $\delta$ is small enough to make sure (\ref{condition 1}) valid, and if $$e^{\frac{c_3}{\delta^2}}\epsilon < \frac{1}{4}\min\{\delta,\rho\},$$ moreover $\epsilon$ is small enough such that (\ref{condition3}) holds, then there exists a transformation
\begin{equation*}
\mathfrak{U}:\left \{
\begin{array}{lc}
x = \xi + u(\xi,\eta),&\\[0.2cm]
y =\eta + v(\xi,\eta),&
\end{array}
\right.
\end{equation*}
where $u$, $v$ belong to $AP_{r-\delta/4,s-\rho/4}(\omega)$ such that $\ \mathfrak{U}^{-1}\mathfrak{M}\mathfrak{U}\ $ takes the form
\begin{equation*}
\left \{
\begin{array}{lc}
\xi_1 = \xi + \eta + f_+(\xi,\eta),&\\[0.2cm]
\eta_1 =\eta + g_+(\xi,\eta),&
\end{array}
\right.
\end{equation*}
and the new mapping is defined in $D(r^+,s^+)$.
Moreover one has the following estimates
\begin{equation*}
\|\mathfrak{U}-id\|_{r^+,s^+} < e^{\frac{2c_3}{\delta^2}} \cdot \epsilon,
\end{equation*}
\begin{equation*}
\|\partial\mathfrak{U}-Id\|_{r^+,s^+} < \frac{1}{\rho} e^{\frac{2c_3}{\delta^2}} \cdot \epsilon,
\end{equation*}
\begin{equation*}
\|f_+\|_{r^+,s^+}+\|g_+ \|_{r^+,s^+}\leq c_6\left\{e^{\frac{c_5}{\delta^2}}\cdot\left( \frac{\epsilon^2}{s}+s\epsilon\right)+\left(\frac{s^+}{s}\right)^2\epsilon \right\}.
\end{equation*}
\label{KAM step}
\end{lemma}

\section{Iteration and proof of the main theorem}

Lemma \ref{KAM step}, known as the iteration lemma in the KAM type proof, is the base of the main theorem. One can use it for infinite times to construct a sequence of transformations and finally eliminate the perturbations. This idea can be dated back to the works of Kolmogorov \cite{Kolmogorov}, Arnold \cite{VIA} and Moser \cite{Moser}, which give this method a name ``KAM". Obviously, it is necessary to make sure that the conditions of Lemma \ref{KAM step} are satisfied each time, and note that the domain will degenerate in the limit case.

Denote the initial $\mathfrak{M}$ by $\mathfrak{M}_0$ and the related domain is
$$D_0:\ |\Im~x|<r_0, \ |y-\alpha|<s_0.$$
By the assumption,
$$\|f\|_{r_0,s_0}+\|g\|_{r_0,s_0} < \epsilon_0.$$

Now we choose the sequences of $r_n$, $s_n$ and $\epsilon_n$ for $n \geq 0$.
Define $\delta_n=\frac{3r_0}{\pi^2}\cdot\frac{1}{n^2},$ thus $\sum_{n=1}^{\infty} \delta_n = r_0/2$, and let $r_n = r_0 - \sum_{k=1}^{n} \delta_k,$ when $n=0$ the sum is 0. Set $s_n = (\epsilon_n)^{2/3}$, $\epsilon_{n+1} = c_7^{(n+1)^4} \cdot \epsilon_n^{4/3}$,
 and
 \begin{equation*}
 e_n=c_7^{6n^4+84n^3+882n^2+6132n+21315}\cdot \epsilon_n,
 \end{equation*}
 then
 \begin{equation*}
 e_{n+1} = c_7^{7(n+1)^4+84(n+1)^3+882(n+1)^2+6132(n+1)+21315}\epsilon_n^{4/3},
 \end{equation*}
and $e_{n+1} \leq e_n^{4/3}$.

Assume that $e_0<1$, since $e_n$ decrease sup-exponentially, then
\begin{equation*}
\left(\frac{s_{n+1}}{s_n}\right) ^2 \cdot \epsilon_n= \left(\frac{\epsilon_{n+1}}{\epsilon_n}\right)^{1/3} \cdot  \epsilon_{n+1} \ll c_7^{-1/3} \cdot  \epsilon_{n+1},
\end{equation*}
and
\begin{equation*}
\begin{aligned}
\|f_+\|_{r^{n+1},s^{n+1}}+\|g_+ \|_{r^{n+1},s^{n+1}} \leq& c_6\left\{e^{c_5\frac{\pi^4}{9r_0^2}(n+1)^4}\cdot\left( \frac{\epsilon_n^2}{s_n}+s_n\epsilon_n\right)+\left(\frac{s_{n+1}}{s_n}\right)^2\cdot \epsilon_n\right\} \\
\leq& c_6\left\{e^{c_5\frac{\pi^4}{9r_0^2}(n+1)^4}\cdot\big(1+\epsilon_n^{1/3}\big)\cdot c_7^{-(n+1)^4}+c_7^{-1/3}\right\}\\
&\times \epsilon_{n+1}.
\end{aligned}
\end{equation*}
By taking $c_7$ large, the coefficient of $\epsilon_{n+1}$ can be made less than 1. Furthermore, one can choose $\epsilon_0$ small enough such that $s_{n+1} < s_n /3$, $\rho_n = s_n- s_{n+1}< \delta_n$ and $\epsilon_n<\rho_n/4$ and $e^{\frac{c_3}{\delta_n^2}}\cdot \epsilon_n < s_n$.
Thus the conditions for the iteration lemma can be satisfied for each step.

Now denote $\mathfrak{V}_n$ the composition of the transformations $\mathfrak{U}_0\circ\mathfrak{U}_1\circ\cdots\circ\mathfrak{U}_n$, then $\mathfrak{M}_{n+1}= \mathfrak{V}^{-1}_n\mathfrak{M}_0\mathfrak{V}_n$. Assume $\mathfrak{V}_n$ takes the form
\begin{equation*}
\left\{
\begin{aligned}
x=\xi+p_n(\xi,\eta) , \\[0.2cm]
y=\eta+q_n(\xi,\eta),
\end{aligned}\right.
\end{equation*}
where $p$, $q\in AP_{r_{n+1},s_{n+1}}(\omega)$, then we have
\begin{equation*}
\left\{
\begin{aligned}
p_n(\xi,\eta) &= u_n(\xi,\eta)+p_{n-1}(\xi+u_n,\eta+v_n),\\[0.2cm] 
q_n(\xi,\eta) &= v_n(\xi,\eta)+q_{n-1}(\xi+u_n,\eta+v_n).
\end{aligned}\right.
\end{equation*}

Note that $s_n \rightarrow 0$ and $r_n \rightarrow \frac{r_0}{2}$, denote $D_{\infty}$ the degenerated domain $|\Im~\xi|<\frac{r_0}{2},\eta=\alpha$, and denote
\begin{equation*}
\Delta_n = \|p_{n}(\xi,\eta)\|_{r_{n+1},s_{n+1}}+\|q_n(\xi,\eta)\|_{r_{n+1},s_{n+1}}.
\end{equation*}
Then by the induction, using the estimates in the iteration lemma we have
\begin{equation*}
\Delta_n \leq \Delta_{n-1} \left(1+\frac{\epsilon_0^{1/9}}{3^n}\right
) + \frac{\epsilon_0^{1/9}}{3^n}.
\end{equation*}
Hence $p_n(\xi,\alpha)$, $q_n(\xi,\alpha)$ converge with respect to the $\|\cdot\|_{r_0/2}$ norm, thus converge uniformly in $D_{\infty}$.
Denote the limit of $p_n$ by $u(\xi)$ and the limit of $q_n$ by $v(\xi)-\alpha$, then
\begin{equation*}
\|u\|_{r_0/2}+\|v-\alpha\|_{r_0/2} \leq 4\epsilon_0^{1/9}.
\end{equation*}

To verify $u$, $v$ are the desired functions, taking the limit on both sides of
$\mathfrak{V}_n\mathfrak{M}_{n+1}=\mathfrak{M}_0\mathfrak{V}_n$
yields that
\begin{equation*}
\left\{
\begin{aligned}
&\xi+u(\xi)+v(\xi)+f(\xi+u(\xi),v(\xi)) = \xi + \alpha + u(\xi+\alpha),\\[0.2cm]
&v(\xi) + g(\xi+u(\xi),v(\xi)) = v(\xi+\alpha).
\end{aligned}\right.
\end{equation*}
Thus the curve $x=\xi+u(\xi)$, $y=v(\xi)$ is invariant under $\mathfrak{M}_0$ and the mapping restricted on this curve is given by the rotation $\xi_1=\xi+\alpha$. The proof of the main theorem is accomplished.

\section{Application}

In this section, as an application, we consider the following pendulum-type equation
\begin{equation}
\ddot{x} + G_x(t,x) = p(t), \label{pendul}
\end{equation}
where $G\in C^{\infty}(\R\times\mathbb{T})$ is $2\pi$ periodic in $x$, and for each fixed $x$, $G(\cdot,x) \in AP(\omega)$, $p\in AP(\omega)$ is a real analytic almost periodic function, and $\omega$ is the Diophantine frequency vector satisfying (\ref{Dio}).

In the periodic case, You \cite{you} had proved that the boundedness of all solutions and the existence of quasi-periodic solutions if and only if the mean value of $p$ is 0. In the quasi-periodic case, Cong, Liang and Han \cite{cong1} also had proved the same conclusion.

Equation (\ref{pendul}) is equivalent to the system

\begin{equation*}
\left\{
\begin{aligned}
&\dot{x}=y ,\\[0.2cm]
 &\dot{y}= -G_x(t,x) + p(t),
 \end{aligned}\right.
\end{equation*}
where the Hamiltonian takes the form $H(x,y)=\frac{y^2}{2}+G(t,x)-xp(t)$.

\subsection{The coordination transformation}
Firstly, under the following coordination transformation
\begin{equation*}
\left\{
\begin{aligned}
&u=\frac{\partial S}{\partial v}=x,\\[0.2cm]
&y=\frac{\partial S}{\partial x}=v+\int_{0}^{t} p(s) ds,
\end{aligned}\right.
\end{equation*}
where $S(x,v,t)=xv+x\int_{0}^{t} p(s)ds $, the system becomes to
\begin{equation*}
\left\{
\begin{aligned}
&\dot{u}= v+P(t),\\[0.2cm]
&\dot{v}=-G_{x}(t,u),
\end{aligned}\right.
\end{equation*}
where $$P(t) = \int_{0}^{t} p(s) ds +C,$$ $C$ will be determined later, and the Hamiltonian is $$h= H\circ\Phi +\frac{\partial S}{\partial t} = \frac{(v+P(t))^2}{2}+G(t,u).$$

Because $p(t) \in AP(\omega)$, assume that the average $ \lim_{T \rightarrow \infty} \frac{1}{T}\int_{0}^{T} p(t) dt = 0,$ and by the Diophantine condition of $\omega$, we have $P(t) \in AP(\omega)$. Choose $C$ such that the average $\lim_{T \rightarrow \infty} \frac{1}{T} \int_{0}^{T} P(s) ds=0,$ that is,
$$
C=\lim_{T \rightarrow \infty} \frac{1}{T} \int_{0}^{T} \int_{0}^{s} p(\tau) d\tau ds,
$$
then the integral $\int_{0}^{t} P(s) ds$ is also almost periodic in t and belongs to $AP(\omega)$.

Now the roles of angle $u$ and time $t$ is changed when $h$ is large sufficiently, then the system is transformed into
\begin{equation*}
\left\{
\begin{aligned}
&\frac{dt}{du}= 2^{-1/2} (h-G(t,u))^{-1/2} ,\\[0.2cm]
&\frac{dh}{du}= p(t)+2^{-1/2} (h-G(t,u))^{-1/2} G_{t}(t,u).
\end{aligned}\right.
\end{equation*}

Define the generator function $$T(\rho,t)= \rho t + 2^{1/2}\rho^{1/2} \int_{0}^{t} P(s) ds,$$ and the corresponding transformation is
\begin{equation*}
\left\{
\begin{aligned}
&\theta = t + 2^{-1/2}\rho^{-1/2} \int_{0}^{t} P(s) ds ,\\[0.2cm]
&h = \rho +2^{1/2}\rho^{1/2}P(t).
\end{aligned}\right.
\end{equation*}

Under this transformation, the system becomes to
\begin{equation*}
\left\{
\begin{aligned}
&\frac{d\theta}{du}= 2^{-1/2}(1+2^{-1/2}\rho^{-1/2}P(t))\cdot(\rho+2^{1/2}\rho^{1/2}P(t)-G(t,u))^{-1/2}+M_2(t,\rho,u) ,\\[0.2cm]
&\frac{d\rho}{du}= (1+2^{-1/2}\rho^{-1/2}P(t))^{-1}\cdot M_1(t,\rho,u),
\end{aligned}\right.
\end{equation*}
where
\begin{equation*}
\begin{aligned}
 M_1=&\frac{1}{2}p(t)\int_{0}^{1} (1+s\rho^{-1}(2^{1/2}\rho^{1/2}P(t)-G(t,u)))^{-3/2}ds\\[0.2cm] &\times\rho^{-1}(2^{1/2}\rho^{1/2}P(t)-G(t,u)) \\[0.2cm] &+
 2^{-1/2}(\rho+2^{1/2}\rho^{1/2}P(t)-G(t,u))^{-1/2}G_t(t,u),
\end{aligned}
\end{equation*}
and $$M_2 = -2^{-3/2}\rho^{-3/2}\int_{0}^{t} P(s)ds \cdot \frac{d\rho}{du}.$$

From Lemma \ref{lemma of function space}, when $\rho$ is large enough, $t= \theta + q(\theta)$ for some $q(\theta) \in AP(\omega)$. For simplicity, by the straight calculation, one can denote the system by
\begin{equation*}
\left\{
\begin{aligned}
&\frac{d\theta}{du}= 2^{-1/2}\rho^{-1/2}+O(\rho^{-3/2}) ,\\[0.2cm]
&\frac{d\rho}{du}= O(\rho^{-1/2}).
\end{aligned}\right.
\end{equation*}
Note that the system is period in $u$ and almost periodic in $\theta$, and the size estimate $O(\cdot)$ is in the sense of the norm $\sup_{u \in [0,2\pi]}\|F(u,\cdot,\cdot)\|_{r,s}$ for some positive $r$, $s$, thus in the sense of the sup-norm.

\subsection{The Poincar\'{e} map}
By the contraction principle, if $\rho$ is large enough, then $\theta$, $\rho$ exist for $u \in [0,2\pi]$ and the Poincar\'{e} map can be written as
\begin{equation}
\left\{
\begin{aligned}
&\theta_1= \theta+2^{-1/2}\rho^{-1/2}+O(\rho^{-3/2}) ,\\[0.2cm]
&\rho_1= \rho+O(\rho^{-1/2}).
\end{aligned}\right. \label{MAP}
\end{equation}
The perturbations is almost periodic in $\theta$, more specifically speaking, the perturbation belongs to $AP(\omega)_{r,s}$ for some positive $r$, $s$.

Let $\delta \mu = 2^{-1/2} \rho^{-1/2}$, $\mu \in [a,b]$, clearly $\rho$ tends to $\infty$ if and only if $\delta$ tends to zero. Under this transformation, the mapping becomes to

\begin{equation}
\left\{
\begin{aligned}
&\theta_1= \theta+\delta\mu+O(\delta^3\mu^3) ,\\[0.2cm]
&\mu_1= \mu +O(\delta^3\mu^4).
\end{aligned}\right.\label{MAP'}
\end{equation}

To check the intersection property of the mapping (\ref{MAP'}), as the intersection property is preserved under coordination transformations, it is sufficient to verify the exact symplectic condition for (\ref{MAP}). On one hand, it is symplectic because it is the time-one map of a Hamiltonian system. Therefore it is sufficient to verify the exact condition
$$\lim_{T \rightarrow \infty} \frac{1}{2T} \int_{-T}^{T} \rho_1d\theta_1 - \rho d\theta = 0.$$
It is known that the limit above is independent of $\rho$ in the domain, thus the integral tends to zero if $\rho$ tend to $\infty$. Up to now, the intersection property is verified.

On the other hand, if the mean value of $p$ is not zero, consider the transformation
\begin{equation*}
\left\{\begin{aligned}
&u= x + v^{-2} \int_{0}^{x}G(t,s)ds,\\[0.2cm]
&y=v-v^{-1}G(t,x),
\end{aligned}\right.
\end{equation*}
where $u$, $x$ belong to $\mathbb{T}$. Under this transformation, the system becomes to

\begin{equation*}
\left\{
\begin{aligned}
&u'= u + O(v^{-2}),\\[0.2cm]
&v'=p(t) + O(v^{-1}).
\end{aligned}\right.
\end{equation*}

Without loss of generality, assume $$p^*=\lim_{T \rightarrow \infty} \frac{1}{T} \int_{0}^{T} p(s) ds > 0.$$
Therefore when $v_0$ is large enough, the solution with $u(0) = u_0$, $v(0)= v_0$ satisfies
\begin{equation*}
\begin{aligned}
v(t) &\geq v_0 + \frac{1}{2} p^* t + \int_{0}^{t} (p(s) -p^*)ds.
\end{aligned}
\end{equation*}
Since $\int_{0}^{t} (p(s) -p^*)ds $ is almost periodic, the solution is unbounded.

\subsection{The main result}

\begin{theorem}
Assume that $G\in C^{\infty}(\R\times\mathbb{T})$ is $2\pi$ periodic in $x$, and for each fixed $x$, $G(\cdot,x) \in AP(\omega)$, $p\in AP(\omega)$, where $\omega$ is the Diophantine frequency vector satisfying (\ref{Dio}), then all solutions of (\ref{pendul}) are bounded and there are infinitely many almost periodic solutions if and only if the average $\lim_{T \rightarrow \infty} \frac{1}{T} \int_{0}^{T} p(s) ds = 0$.\label{pend thm}
\end{theorem}

\begin{proof}
	If the mean value of $p$ is zero, for small enough $\delta$, the small twist theorem \ref{small twist} guarantees the existence of infinitely many invariant curves for the mapping  (\ref{MAP'}) and thus for the mapping (\ref{MAP}). Therefore the system (\ref{pendul}) has infinitely many almost periodic solutions, and all solutions are bounded. Conversely if the mean value of $p$ is not zero, the discussion above shows that the solution with big enough initial value $y_0$ is unbounded.
\end{proof}

\begin{remark}
	It follows from the proof of Theorem \ref{pend thm} that if the conditions of the theorem hold, then system (\ref{pendul}) has infinitely many almost periodic solutions with frequencies $\{\omega,\frac{1}{\delta\alpha}\}$ satisfying
	$$\left|(l,\omega)\frac{\delta\alpha}{2\pi}-n\right|>\prod_{i=1}^{\infty}\frac{\gamma}{1+|l_i|^{2+2\mu}\cdot i^{2+2\mu}}, ~ \alpha \in [a+\gamma,b-\gamma]$$
	 for any $n \in \Z$, $l \in Z^{\N^+}_{\star}$ and fixed positive $\gamma$, $\delta$.
\end{remark}

\section{Appendix}\label{appen}
\subsection{The proof of Proposition \ref{modified module containment}} \label{1}

We prove the sufficiency firstly. From Lemma \ref{translation2},
	for all $ \epsilon>0$, there exists a finite set $$\{\lambda_n:n=1,...,N\}\subseteq\mathcal{M}(h)\subseteq\mathcal{M}(f,g),$$ and a $\delta>0 $, such that the set $$S=\{\tau:|\lambda_n\tau|<\delta \mod(2\pi),n=1,...,N\}$$ belongs to $T_{\epsilon}(h)$.
	
Because $\{\lambda_n\}$ are finite linear combinations of exponents of $f$ and $g$,  from Lemma \ref{translation1}, there exists a $\delta'>0$ such that
	 $$T_{\delta'}(f)\bigcap T_{\delta'}(g)\subseteq S.$$
	  Note that this set is not empty. If $\{f(t+t_n)\}$ and $\{g(t+t_n)\}$ are Cauchy sequences in uniform topology, there exists $N$ such that for all $ n,m\geq N$ and $t\in\R$, we have
\begin{equation*}
\left\{\begin{aligned}
&|f(t+tn-tm)-f(t)|<\delta', \\[0.2cm]
&|g(t+tn-tm)-g(t)|<\delta'.
\end{aligned}\right.
\end{equation*}
Therefore we have $|h(t+tn-tm)-h(t)|<\epsilon$ for all $ n,m\geq N$, thus $\{h(t+t_n)\}$ converges uniformly.
	
	The necessity is proved by contradiction. Assume that there is an exponent of $h$, denoted by $\lambda_0$, which does not belong to $\mathcal{M}(f,g)$, then there are two situations:\\[0.2cm]
	(i) for all $ l\geq1$, we have $l\cdot\lambda_0 \notin \mathcal{M}(f,g)$;\\[0.2cm]
	(ii) there exists $ l_0 \geq 2$ such that $l_0 \cdot \lambda_0 \in \mathcal{M}(f,g)$.

	Firstly, from Lemma \ref{translation1} and Lemma \ref{translation2}, it is clear that for all $\epsilon>0$, there exists a $ \delta>0$ such that $$T_{\delta}(f)\bigcap T_{\delta}(g)\subseteq T_{\epsilon}(h).$$
	 Then for all $ \epsilon>0$, there exists $ \{\lambda_n:n=1,...,N\}\subseteq \mathcal{M}(f,g)$ and a $\delta>0$ such that $$\{\tau:|\lambda_n\tau|<\delta \mod(2\pi),n=1,...,N\} \subseteq \{\tau:|\lambda_0\tau|<\epsilon \mod (2\pi)\}.$$

In the first situation, set $\theta_n=0,n=1,...,N$ and $\theta_0=\pi$. It is clear that if $\{k_n\in \Z\}$ satisfies that $\sum_{n=0}^{N}k_n\lambda_n=0$, then we have $k_n\equiv0$ for $n=0,...,N$. Thus $\sum_{n=0}^{N}k_n\theta_n=0$. From Kronecker Theorem \ref{kronecker}, there exists a $\tau$ such that both $|\lambda_n\tau|<\delta \mod(2\pi)$ and $|\lambda_0\tau-\pi|<\delta \mod (2\pi)$ are valid for $n=1, ... ,N$. Then we have $|\lambda_0\tau-\pi|<\delta$ and $|\lambda_0\tau|<\epsilon$ at the same time. Set $\epsilon$ and $\delta$ small enough, then such $\tau$ does not exist, this makes a contradiction.

In the second situation, it suffices to set $\theta_0=\frac{2\pi}{l_0}$. If $\sum_{n=0}^{N}k_n\lambda_n=0$, then $l_0|k_0$ and  $$\sum_{n=0}^{N}k_n\theta_n=0 \mod(2\pi).$$ Similar to the first situation, by kronecker theorem, this makes a contradiction.

\subsection{The proof of Proposition \ref{Dio}}\label{2}
	Consider the product space $\Omega=[0,1]^{\N^+}$ equipped with product probability measure $m$. Denote $\D_{\gamma_0}$ the  subset of $\{\omega\} \subseteq \Omega$ satisfying (\ref{Dio}), and denote $|\mbox{supp} ~l|$ the number of elements of $\{i: l_i \neq 0\}$. Denote $c_n$ the volume of n-dimensional unit ball, then $c_n \leq C$, where $C$ is an absolutes constant. Then we have
	\begin{equation*}
\begin{aligned}
m(\Omega-\D_{\gamma_0}) \leq \ & C\sum_{l\in Z^{\N^+}_{\star}} \frac{\gamma_0\cdot\sqrt{2}^{|\mbox{supp}\, l| -1}}{\prod_{i=1}^{\infty}(1+|l_i|^{1+\mu}\cdot i^{1+\mu})}\\
\leq \ & C \cdot \gamma_0 \cdot \prod_{i=1}^{\infty}\left( \sum_{l_i=-\infty \atop l_i \neq 0}^{\infty} \frac{\sqrt{2}}{1+i^{1+\mu}\cdot |l_i|^{1+\mu}}+1\right) \\
\leq \ & C \cdot \gamma_0 \cdot \prod_{i=1}^{\infty}\left(1+\frac{\sqrt{2}C(\mu)}{i^{1+\mu}}\right)\\
= \ & O(\gamma_0)
\end{aligned}
\end{equation*}
	When $\gamma_0$ tends to zero, $m(D_{\gamma_0})$ tends to 1, thus such $\omega$ exists.

\subsection{The proof of Proposition \ref{Dio2}} \label{3}

	Denote $\mathcal{R}_{\omega}^{l,n}$ the set $$\left\{\alpha \in [a+2\pi\gamma,b-2\pi\gamma]:\left|(\omega,l)\frac{\alpha}{2\pi}-n\right|\leq \gamma \prod_{i=1}^{\infty} \frac{1}{1+|l_i|^{2+2\mu}\cdot i^{2+2\mu}}\right\},$$ i.e. 	when $(\omega,l)>0$,
	
	\begin{equation*}
	\left\{
	\begin{aligned}
	&\frac{2\pi n}{|(\omega,l)|}-\frac{2\pi\gamma}{|(\omega,l)|} \prod_{i=1}^{\infty} \frac{1}{1+|l_i|^{2+2\mu}\cdot i^{2+2\mu}}  \leq \alpha ,\\  & \frac{2\pi n}{|(\omega,l)|} + \frac{2\pi\gamma}{|(\omega,l)|} \prod_{i=1}^{\infty} \frac{1}{1+|l_i|^{2+2\mu} \cdot i^{2+2\mu}} \geq \alpha,
	\end{aligned}\right.
	\end{equation*}
when $(\omega,l)<0$,
	\begin{equation*}
	\left\{
	\begin{aligned}
	&\frac{-2\pi n}{|(\omega,l)|}-\frac{2\pi\gamma}{|(\omega,l)|} \prod_{i=1}^{\infty} \frac{1}{1+|l_i|^{2+2\mu}\cdot i^{2+2\mu}}  \leq \alpha, \\
	&\frac{-2\pi n}{|(\omega,l)|} + \frac{2\pi\gamma}{|(\omega,l)|} \prod_{i=1}^{\infty} \frac{1}{1+|l_i|^{2+2\mu}\cdot i^{2+2\mu}} \geq \alpha.
	\end{aligned}\right.
	\end{equation*}

	The number $n$ should satisfy the following inequalities
	\begin{equation*}
	\begin{aligned}
	\frac{a|(\omega,l)|}{2\pi} \leq n \leq \frac{b|(\omega,l)|}{2\pi},\ \mbox{or}\ \  \frac{-b|(\omega,l)|}{2\pi} \leq n \leq \frac{-a|(\omega,l)|}{2\pi}
	\end{aligned}
	\end{equation*}
to make the intersection of $R_{\omega}^{l,n}$ with $[a+2\pi\gamma,b-2\pi\gamma]$ not empty.
	
Therefore we have
	\begin{equation*}
	\begin{aligned}
	m\left(\bigcup_{l,n}\R_{\omega}^{l,n}\right) \leq& \ 4\pi  \sum_{l\in Z^{\N^+}_{\star}} \sum_n \left( \frac{\gamma}  {|(\omega,l)|} \cdot \prod_{i=1}^{\infty}  \frac{1}{1+|l_i|^{2+2\mu}\cdot i^{2+2\mu}} \right )\\
	\leq& \ 4\pi  \sum_{l\in Z^{\N^+}_{\star}}\frac{\gamma} {|(\omega,l)|}\left(\frac{(b-a)|(\omega,l)|}{\pi}+2\right)\cdot  \prod_{i=1}^{\infty}  \frac{1}{1+|l_i|^{2+2\mu}\cdot i^{2+2\mu}} \\
	\leq& \ 4\pi \sum_{l\in Z^{\N^+}_{\star}} \frac{(b-a)\gamma}{\pi}\prod_{i=1}^{\infty}\frac{1}{1+|l_i|^{2+2\mu}\cdot i^{2+2\mu}}\\ &+\sum_{l\in Z^{\N^+}_{\star}} 8\pi\gamma\prod_{i=1}^{\infty}\frac{1+|l_i|^{1+\mu}\cdot i^{1+\mu}}{1+|l_i|^{2+2\mu}\cdot i^{2+2\mu}}\\
	\leq& \ 4\pi \frac{(b-a)\gamma}{\pi}\prod_{i=1}^{\infty} \left(1+\sum_{l_i=-\infty \atop l_i \neq 0} \frac{1}{1+|l_i|^{2+2\mu}\cdot i^{2+2\mu}}\right) \\
	& +8\pi\gamma \prod_{i=1}^{\infty} \left(1+\sum_{l_i=-\infty\atop l_i \neq 0} \frac{1}{1+|l_i|^{1+\mu}\cdot i^{1+\mu}}\right) \\
	=& \ O(\gamma)
	\end{aligned}
	\end{equation*}
	
	 When $\gamma$ tends to zero, the measure of the set of $\alpha$ which satisfies (\ref{Dio2}) tends to full measure, thus this kind of $\alpha$ exists.

\subsection{Proof of Lemma \ref{lemma of function space}}\label{4}

The property (i) is obvious, thus we strat with (ii).

Assume $f , u , v\in AP_{r,s}(\omega)$. Firstly note that $AP_{r,s}(\omega)$ is a Banach space for any positive $r,s$.
Assume $\epsilon =\|u\|_{r,s} + \|v\|_{r,s}$ which are less than $\min\{r-r',s-s'\}$, therefore $h(\xi,\eta) := f(x+u(\xi,\eta),y+v(\xi,\eta))$ are well defined in $D(r',s')$ and analytic in $\xi$ and $\eta$.

From Proposition \ref{modified module containment} and its corollary, it is known that for all $\eta$, $h(\cdot,\eta)$ is an almost periodic function and $\mathcal{M}(h) \in \mathcal{M}(f,g)$, moreover, $\mathcal{M}(h) \subseteq \mbox{span}\{\omega_i\}$. Thus formally we have $h=\sum_{l\in Z^{\N^+}} h_l(\eta) e^{i(\omega,l)\xi}$.  Unlike the sup-norm, the weighted norm of the function is very sensitive to coordinate transformations. Fortunately, it is enough to consider the canonical transformation that is close to identity.

Formally, $h$ also equals to

\begin{equation*}
h(\xi,\eta)=\sum_{l\in Z^{\N^+}} f_l(\eta+v) \cdot e^{i(\omega,l)\xi} \cdot e^{i(\omega,l)u(\xi,\eta)}.
\end{equation*}
From Proposition \ref{modified module containment} and its corollary again, we have $e^{i(\omega,l)u(\xi,\eta)} \in AP_{r,s}(\omega).$ From Lemma \ref{homogeneous} there is the estimate

\begin{equation*}
 \|e^{i(\omega,l)u(\xi,\eta)}\|_{r,s} \leq \sum_{k=0}^{\infty} \frac{\epsilon^k}{k!} = e^{\|l\|\epsilon}.
\end{equation*}

Meanwhile,
\begin{equation*}
\begin{aligned}
f_l(\eta + v(\xi,\eta)) &= \sum_{k=0}^{\infty} \sum_{j=0}^{k}\frac{f_{l,k}(\alpha)}{k!}  \begin{pmatrix}
k \\ j
\end{pmatrix}(\eta-\alpha)^{k-j} v(\xi,\eta)^{j} \\
&= \sum_{j=0}^{\infty} v(\xi,\eta)^{j} \cdot \sum_{k=j}^{\infty}\frac{f_{l,k}(\alpha)}{k!} \begin{pmatrix}
k \\ j
\end{pmatrix} (\eta-\alpha)^{k-j} .
\end{aligned}
\end{equation*}
By Cauchy's estimate \ref{Cauchy},
\begin{equation*}
|f_{l,k}(\alpha)| \leq \frac{|f|_{s}}{s^k}.
\end{equation*}
Thus

 \begin{equation*}
 \begin{aligned}
 \Biggl\|v(\xi,\eta)^{j} \sum_{k=j}^{\infty}\frac{f_{l,k}(\alpha)}{k!} \begin{pmatrix}
 k \\ j
 \end{pmatrix} (\eta-\alpha)^{k-j}\Biggr\|_{r',s'} \leq&  \|v(\xi,\eta)\|_{r',s'} ^{j}
  \cdot |f|_s \\ &\times \sum_{k=j}^{\infty} \frac{1}{k!} \begin{pmatrix}
  k \\ j
  \end{pmatrix} s^{-k} s'^{k-j},
 \end{aligned}
 \end{equation*}
 and
\begin{equation*}
\begin{aligned}
\|f_l\|_{r',s'} &\leq |f|_s \cdot \sum_{k=0}^{\infty} s^{-k} \sum_{j=0}^{k}\frac{1}{k!} \begin{pmatrix}
k \\ j
\end{pmatrix} s'^{k-j} \epsilon^{j}\\
&\leq |f|_s \cdot \sum_{k=0}^{\infty} \frac{s^{-k}}{k!} (s'+\epsilon)^k\\
&\leq e^{\frac{s'+\epsilon}{s}}|f|_s.
\end{aligned}
\end{equation*}

Then we have

\begin{equation*}
\|h\|_{r',s'} \leq \|f\|_{r,s} \cdot e^{\epsilon+\frac{s'+\epsilon}{s}}.
\end{equation*}

The property (iii) is proved by the well-known implicit function theorem in Banach space, it can be found in the book \cite{nonlinear} for example.

To solve the $u'(x,y)$ and $v'(x,y)$ from (\ref{inverse mapping}), it suffices to consider the equations

\begin{equation*}
\left\{
\begin{array}{lc}
0 = u'(x,y) + u(x+u'(x,y),y+v'(x,y)), &  \\[0.2cm]
0 = v'(x,y) + v(x+u'(x,y),y+v'(x,y)). &
\end{array}
\right.
\end{equation*}
Thus we define $G:AP_{r',s'}(\omega)^2 \times AP_{r,s}(\omega)^2 \rightarrow AP_{r',s'}^2$ by

\begin{equation*}
G(u',v',u,v):=
\begin{pmatrix}
u'(x,y) + u(x+u'(x,y),y+v'(x,y))  \\[0.2cm]
v'(x,y) + v(x+u'(x,y),y+v'(x,y))
\end{pmatrix}.
\end{equation*}
when $\|u'\|_{r',s'}+\|v'\|_{r',s'} < \min(r-r',s-s')$, $G$ is well defined.

To prove the existence of $u',v'$ satisfying $G=0$, for small enough $u,v$, by the implicit function theorem, it is sufficient to verify\\[0.2cm]
(a) $G(0,0,0,0)=0$;  ~\\[0.2cm]
(b) $\frac{\partial(G_1,G_2)}{\partial(u',v')}$ is invertible; ~\\[0.2cm]
(c) $G_{u'},G_{v'}$ is continuous in $(u,'v,'u,'v)$.

By the straight calculation, we have
\begin{equation*}
	\begin{pmatrix} G_{1,u'} & G_{1,v'} \\[0.2cm] G_{2,u'} & G_{2,v'}\end{pmatrix} \begin{pmatrix}
	h_1 \\[0.2cm] h_2
	\end{pmatrix}
	 = \begin{pmatrix} h_1+u_{x}(x+u',y+v')h_1 + u_{y}(x+u',y+v')h_2 \\[0.2cm] v_{x}(x+u',y+v')h_1 + h_2+v_{y}(x+u',y+v')h_2\end{pmatrix},
\end{equation*}
where $h_1$, $h_2\in AP_{r',s'}(\omega)$, and obviously when $(u',v',u,v)=(0,0,0,0)$, the operator above is the identity operator.

Therefore by Cauchy's estimate and the property (ii) which is just proved, when $\|u\|_{r,s}+\|v\|_{r,s}$ and $\|u'\|_{r',s'}+\|v'\|_{r',s'}$ small enough, the operator above is close to the identity operator $I$, which maps from $AP_{r',s'}$ to itself, in the sense of operator norm. The proof of the continuation in $(u',v',u,v)$ is similar. Thus the property (iii) is proved by the implicit function theorem in Banach space. In fact from the proof of the implicit function theorem we get that, if assume

\begin{equation*}
\begin{aligned}
	&\|u\|_{r,s}+\|v\|_{r,s} \leq \epsilon,
\end{aligned}
\end{equation*}
where $\epsilon$ is small enough to satisfy

\begin{equation*}
\max\left\{\frac{1}{r-r'},\frac{1}{s-s'}\right\} \cdot \epsilon \cdot \exp\left(2\epsilon+\frac{s'+2\epsilon}{s} \right) < 1/2,
\end{equation*}
then $u'$, $v'$ can be solved and satisfy $\|u'\|_{r,s}+\|v'\|_{r,s} \leq \epsilon$.

\subsection{Cauchy's estimate} \label{Cauchy}

\begin{lemma}
	Assume the analytic function $f$ defined in $B_s(z)$  satisfies $|f| \leq M$, then it satisfies
	\begin{equation*}
	|f'(w)| \leq \frac{M}{s-s'},
	\end{equation*}
	where $w \in B_{s'}(z)$ and $s'<s$.
\end{lemma}

\begin{corollary}
	Assume $f(x,y) \in AP_{r,s}(\omega)$, then it satisfies
	\begin{equation*}
\|f_y\|_{r,s'} \leq \frac{\|f\|_{r,s}}{s-s'}.
	\end{equation*}
\end{corollary}

\begin{lemma}
	Assume $f(x,y) \in AP_{r,s}(\omega)$, then it satisfies
	\begin{equation*}
	\|f_x\|_{r',s} \leq \frac{\|f\|_{r,s}}{r-r'}.
	\end{equation*}
\end{lemma}

\begin{proof}
Assume	$f$ takes the form $f(x,y) = \sum_{l\in Z^{\N^+}} f_l(y)e^{i(\omega,l)x}$, where $|\Im~ x|<r$, $y \in B_s(y_0)$, thus

	\begin{equation*}
		f_x = \sum_{l\in Z^{\N^+}} (\omega,l) \cdot f_l(y)e^{i(\omega,l)x}.
	\end{equation*}
	
Then we have
\begin{equation*}
\begin{aligned}
\|f_x\|_{r',s} &\leq  \sum_{l\in Z^{\N^+}} \bigg(\sum_{i=1}^{\infty}|l_i|\bigg) \cdot e^{-\delta\|l\|} \cdot  |f_l(y)| e^{r\|l\|}\\
&\leq \|f\|_{r,s} \cdot \sup_{l\in Z^{\N^+}} \left\{ \left(\sum_{i=1}^{\infty}|l_i| \right)e^{-\delta\|l\|} \right\}\\
& \leq \frac{e^{-1}}{\delta} \|f\|_{r,s}\\
& \leq \frac{\|f\|_{r,s}}{r-r'}.
\end{aligned}
\end{equation*}
\end{proof}

\section*{References}
\bibliographystyle{elsarticle-num}

\end{document}